\documentclass[a4paper,oneside,11pt]{scrartcl}
\input{mypaper.sty}
\def\csat {c_{\textnormal{sat}}}

\usepackage{bm}

\begin{document}
\title{Adaptive $\mathcal{H}$-Matrix Computations in Linear Elasticity}
\author{M.\ Bauer and M.\ Bebendorf\footnote{Faculty of Mathematics, Physics and Computer Science, University of Bayreuth, 95447 Bayreuth, Germany}}
\date{\today}
\maketitle

\begin{abstract}
This article deals with the adaptive and approximative computation of the Lam\'e equations. The equations of linear elasticity are considered as boundary integral equations and solved in the setting of the boundary element
method (BEM). Using BEM, one is faced with the solution of a system of equations with a fully populated system matrix, which is in general very costly. Some adaptive algorithms based on hierarchical matrices and the adaptive cross approximation 
%which are normally the techniques to overcome this problem
 are proposed. At first, an adaptive matrix-vector multiplication scheme is introduced for the efficient treatment of multiplying discretizations with given data. The strategy, to reach this aim,  is to use error estimators and techniques known from adaptivity. 
 %which are originally based on grid refinement. 
 The case of approximating the system matrix appearing in the linear system of equations with this new type of adaptivity is also discussed.
\end{abstract}

\section{Introduction}\label{sec:one}
Almost every solid body deforms under the influence of force. The theory of linear elasticity attempts to mathematically describe deformations of bodies which return to their original shape after the application of force. The
objective here is to determine a displacement field $u(x)$ for all $x$ in a bounded domain $\Omega \subset \mathbb{R}^3$, such that for an elastic body the equilibrium equations
\begin{equation}\label{eq:equilibrium}
- \sum_{j = 1}^3 \frac{\partial}{\partial x_j} \sigma_{ij} (u,x) = f_i(x) \quad \textnormal{for } x \in \Omega, \; i = 1,2,3,
\end{equation}
hold, where a reversible, isotropic and homogeneous material behavior is assumed; see~\cite{hi04, os08, teodrescu13}. Using Hooke's law, the components of the stress tensor
\begin{equation*}
\sigma_{ij}(u,x) = \frac{E \nu}{(1+\nu)(1-2\nu)} \delta_{ij} \sum_{k = 1}^{3} e_{kk}(u,x) + \frac{E}{1+\nu}e_{ij}(u,x) 
\end{equation*} 
 $\textnormal{for } x \in \Omega, \; i = 1,2,3$, are linked to the strain tensor $e_{ij}(u,x)$ having the form
 \begin{equation*}
 e_{ij} = \frac{1}{2} \left[  \frac{\partial}{\partial x_i} u_j(x) + \frac{\partial}{\partial x_j} u_i(x) \right] \quad \textnormal{for } x \in \Omega, \; i ,j = 1,2,3,
 \end{equation*}
 under the assumption of small deformations. The numbers $E > 0$ and $\nu \in (0,1/2)$ denote the Young modulus and the Poisson ratio. After some transformation, see for example~\cite{os08}, we end up with the Navier system
 \begin{equation*}\label{eq:navierSystem}
 -\mu \Delta u(x) - (\lambda + \mu) \textnormal{grad }\textnormal{div }u(x) = f(x) \quad \textnormal{for } x \in \Omega
 \end{equation*}
 with the Lamé constants
 \begin{equation}\label{eq:lameconst}
 \lambda = \frac{E\nu}{(1+\nu)(1-2\nu)} \quad \textnormal{and} \quad \mu = \frac{E}{2(1+\nu)}.
 \end{equation}
Typical boundary conditions in solid mechanics are a mixture of Dirichlet conditions describing fixed restraints and Neumann conditions for free bearings.  For reasons of simplicity we choose one Dirichlet condition 
\begin{equation*}
\gamma_0^{\textnormal{int}} u(x) = g_D(x) \quad \textnormal{for } x \in \Gamma_D
\end{equation*}
and one Neumann boundary condition
\begin{equation*}
\gamma_1^{\textnormal{int}} u(x) = g_N(x) \quad \textnormal{for } x \in \Gamma_N,
\end{equation*}
where $\gamma_0^{\textnormal{int}}$, $\gamma_1^{\textnormal{int}}$ denote the Dirichlet and the Neumann trace operator and $\partial \Omega = \overline{\Gamma}_D \cup \overline{\Gamma}_N$ with $\Gamma_D \cap \Gamma_N = \emptyset$. Additionally, we assume a positive measure of the Dirichlet part, i.e.
\begin{equation*}
\int_{\Gamma_D} \ud s > 0,
\end{equation*}
in order to guarantee the existence of a unique solution of the considered problem.

The method of choice of the numerical solution of the above described problem is the Finite Element Method~(FEM). The resulting system matrix, or more precisely the discretized operator, is sparse. 
However, depending on the underlying grid, the matrices can quickly become very large. Another method, which is not confronted with large system matrices since only the boundary has to be discretized, is the Boundary Element Method~(BEM). 
 Initially this method was not a real alternative for FEM. Since in BEM a non-local operator is discretized,
 which leads to fully populated matrices, an efficient application of BEM was not possible. 
 
The situation changed with the progressive development towards fast boundary element methods. Methods like the fast multipole method or hier\-archical matrices ($\mathcal{H}$-matrices) reduce the complexity by approximating the discretized operator to such an extent that BEM represents an alternative to FEM.  While the fast multipole method~\cite{gr87, r85} was physically motivated and designed for specific problems, hierarchical matrices could be kept more general; see~\cite{hackbusch99, hk00}. As the name hierarchical matrix already suggests, this technique is based on a hierarchical partitioning of the discrete operator into suitable blocks.
Each of these blocks contains a low-rank approximation to the original block entry, with the whole matrix having only a logarithmic-linear stor\-age requirement. The representations are further advantageous in connection with iterative solution methods, which contain many matrix-vector multiplications. Hierarchical matrices offer the possibility to perform fast matrix-vector multiplications of logarithmic-linear complexity. 
Employing only a few of the original matrix entries, the adaptive cross approximation (ACA)~\cite{bebendorf00} has become quite popular to construct the low-rank approximation on suitable blocks. The number of matrix entries was further reduced by adding another level of adaptivity to ACA; see~\cite{bb21}. With the so-called block-adaptive cross approximation (BACA), not every block is approximated in the same way and to the same accuracy as in ACA,
but only those blocks are more accurately approximated that lead the greatest gain in accuracy of the solution.
The aim of this article is to adapt the ideas of BACA to the construction of an adaptive version of the matrix-vector multiplication.
The situation when multiplying a matrix by a vector in some sense is similar to the solution
of linear systems. When multiplying a partitioned matrix with a vector~$x$, not every approximated block has the same effect on the accuracy of the result, especially if the vector to be multiplied contains large clusters of zeros, for instance.
In order to exploit the structural differences in the vector~$x$ and to detect the best blocks, error estimators and techniques known from adaptivity are used. Note that the block-wise low-rank approximations will be successively improved without changing the hierarchical 
block structure or the grid. 

The following topics are considered in the article. Section~\ref{sec:Approx} presents basic techniques for approximation using low-rank matrices. In more detail, partitioning, cluster trees, hierarchical matrices and adaptive cross approximation are briefly discussed.
In Section~\ref{sec:amvm} we introduce an adaptive scheme for an approximate computation of the matrix-vector multiplication. Furthermore, the convergence of the investigated method and some properties of the proposed error estimator are analyzed.
Since the techniques in Sections~\ref{sec:Approx} and~\ref{sec:amvm} can be applied in many situations, we will first discuss a general case before moving on to the boundary integral approximation of the equations of linear elasticity in Section~\ref{sec:BI_LE}. Adapting the ideas of BACA to the case of linear elasticity, i.e. the Lam\'e equations, in Section~\ref{sec:adapt_BACA}, we are able 
to compute linear elasticity in a fully adaptive manner with $\mathcal{H}$-matrices. Finally, numerical examples presented in Section~\ref{sec:numerics} show a performance acceleration and a storage reduction for the numerical computation of the boundary integral formulation of linear elasticity.

\section{Approximation with Low-Rank Matrices}\label{sec:Approx}
We consider matrices $A \in \mathbb{R}^{M \times N}$ having the representation
\begin{equation*}
A = \Lambda_1 \mathcal{A} \Lambda_2^*
\end{equation*}
with a non-local linear operator $\mathcal{A}$ which depends linearly on the bivariate kernel function~$\kappa: \R^d \times \R^d \rightarrow \R^d$. The prototype for such an operator is
\begin{equation*}
(\mathcal{A}v)(x) = \int_{\Omega} \kappa(x,y) v(y) \ud\mu_y, \quad x \in \Omega,
\end{equation*}
where $\Omega \subset \R^d$ is a bounded domain with Lipschitz boundary and $\mu$ denotes the corresponding measure. The operators $\Lambda_1: L^2(\Omega) \rightarrow \R^M$ and $\Lambda_2: L^2(\Omega) \rightarrow \R^N$ are 
assumed to be linear. The adjoint operator $\Lambda_2^*: \R^N \rightarrow L^2(\Omega)$ is defined as
\begin{equation*}
(\Lambda_2^*, f)_{L^2(\Omega)} = z^T(\Lambda_2 f), \quad z \in \R^N, \quad f \in L^2(\Omega).
\end{equation*}
These two operators are used to describe different discretizations. Two examples are:
\begin{enumerate}
\item Galerkin method: Choosing functions $\varphi_i$, $i = 1,\dots,M$, and $\psi_j$, $j = 1,\dots,N$, results in the discretization 
\begin{equation*}
(\Lambda_1f)_i = \int_{\Omega} f(x) \varphi_i(x) \ud \mu_x \quad \textnormal{and} \quad (\Lambda_2 f)_j = \int_{\Omega} f(x) \psi_j(x) \ud \mu_x.
\end{equation*}
\item Collocation method: Choosing points $y_i$, $i = 1,\dots,M$, and functions $\varphi_j$, $j = 1,\dots,N$, leads to
\begin{equation*}
(\Lambda_1f)_i = f(y_i)\quad \textnormal{and} \quad (\Lambda_2 f)_j = \int_{\Omega} f(x) \varphi_j(x) \ud \mu_x.
\end{equation*}
\end{enumerate}

The approximation of $A$ with low-rank matrices can be done by approximating the bivariate function~$\kappa$ with a degenerate function~$\tilde \kappa$, i.e.\ there exists functions $u_l: X \rightarrow \R$ and $v_l: Y \rightarrow \R$, $l = 1,\dots,k$,
such that 
\begin{equation}\label{eq:approx}
\kappa(x,y) \approx \tilde \kappa(x,y) := \sum_{l = 1}^k u_l(x) v_l(y),\quad x\in X,\;y\in Y,
\end{equation}
holds for two domains $X, Y \subset \Omega$. Such an approximation automatically leads to a matrix $\tilde A$ of rank at most~$k$, since with
\begin{equation*}
a_l := \Lambda_1 u_l \in \R^M \quad \textnormal{and} \quad b_l := \Lambda_2 v_l \in \R^N, \quad l = 1,\dots,k,
\end{equation*}
it follows
\begin{equation*}
\tilde A = \Lambda_1 \tilde{\mathcal{A}} \Lambda_2^* = \Lambda_1 \sum_{l = 1}^k u_l b_l^T = \sum_{l = 1}^k (\Lambda_1 u_l) b_l^T = \sum_{l = 1}^k a_l b_l^T.
\end{equation*}
where $\tilde{\mathcal{A}}$ is defined by $(\tilde{\mathcal{A}}v)(x) := \int_\Omega \tilde \kappa(x,y) v(y) \ud \mu_y$. The reversal of the statement is  not true in general. 

A matrix $\tilde A \in \R^{M \times N}$ having rank $k$ is called low-rank matrix if the condition 
\begin{equation*}
k(M+N) < M \cdot N
\end{equation*}
is fulfilled. Using the outer product representation, i.e.\ $\tilde A = UV^T$ with matrices $U \in \R^{M \times k}$ and $V \in \R^{N \times k}$, $\tilde A$ requires $k(M+N)$ instead of $M \cdot N$ units of storage. Additionally, the multiplication 
of $\tilde A$ by a vector~$x$ can be done with $\mathcal O(k(M+N))$ arithmetic operations instead of $\mathcal O(M \cdot N)$. The best rank-$k$ approximation is given by the truncated singular value decomposition; see~\cite{ey1936}.
The advantage of the latter method over kernel approximation~\eqref{eq:approx} is its black-box nature as it relies only on the entries of~$A$. Since it has cubic complexity, the truncated singular value decomposition
cannot be used in practice.  

\subsection{Partitions and Cluster Trees}
Low-rank approximations are typically employed on suitable blocks and not for the whole matrix. In most cases the approximation of the entire matrix is not possible at all. Therefore, the matrix $A \in \R^{M \times N}$ is decomposed 
into blocks $t \times s$, $t \subset I := \{1,\dots,M\}$ and $s \subset J:= \{1,\dots,N\}$ at first.  After that each suitable block $A_{ts}$ is approximated with a low-rank matrix
\begin{equation*}
A_{ts} \approx UV^T, \quad U \in \mathbb{R}^{t \times k}, \; V \in \R^{s \times k},
\end{equation*}
where the number $k$ is small compared to $\lvert t\rvert$ and $\lvert s\rvert$. 
Let $\textnormal{supp } \Lambda_1=X_t$ and $\textnormal{supp } \Lambda_2=X_s$ be the clusters corresponding to
the index sets~$t$ and~$s$, respectively.
For example, in the case of Galerkin discretizations $X_t$ denotes the union of the supports $X_i := \textnormal{supp } \varphi_i$, $i \in t$.
A block is suitable or admissible for approximation if it satisfies the condition 
\begin{equation}\label{eq:adm}
\min \{\textnormal{diam } X_t, \textnormal{diam } X_s\} < \beta \, \textnormal{dist}(X_t, X_s)
\end{equation}
for a given $\beta > 0$.  The expression 
\[
\textnormal{diam }X = \sup_{x,y \in X} \lvert x-y\rvert \quad \textnormal{and} \quad \textnormal{dist}(X,Y) = \inf_{x \in X, y \in Y} \lvert x-y\rvert
\]
are the diameter and the distance of two bounded sets $X, Y \subset \Omega$. Condition~\eqref{eq:adm} guarantees the existence of low-rank approximations if $A$ discretizes an integral representation or the inverse of second-order elliptic partial
differential operators; see~\cite{bebendorf08}.

A partition $P$ of the matrix indices $I \times J$ consisting of admissible blocks or blocks which are small can be found as the leaves of a block-cluster tree~$T_{I\times J}$; see~\cite{hk00, bebendorf08}. This quad-tree can be constructed from two separate binary cluster trees~$T_I$ and~$T_J$ with roots~$I$ and~$J$, respectively.
The sons $S_I(t) = \{ t', t''\} \subset T_I$ of each node $t \in T_I$ (or $s \in T_J$), if they exist, satisfy $t' \cup t''=t$ and $t' \cap t'' = \emptyset$. The leaves of~$T_I$ are gathered in the set $\mathcal{L}(T_I) := \{t \in T_I \, : \, S_I(t) = \emptyset\}$. Applying the mapping~$S_I$ recursively,
a cluster tree~$T_I$ can be constructed consisting of several levels $T_I^{(l)}$, $l = 0,\dots,L$, where $L$ denotes the depth of the tree.
Once both cluster trees $T_I$ and $T_J$ have been generated, the block-cluster tree $T_{I \times J}$ can be constructed by recursively subdividing $I\times J$ by
following the trees $T_I$ for the rows and $T_J$ for the columns until either~\eqref{eq:adm} is satisfied or the clusters cannot be subdivided further.
As a result, the partition~$P$ consists of admissible blocks $P_{\textnormal{adm}}$ and non-admissible blocks $P_{\textnormal{non-adm}}$, i.e.
\[
P := \mathcal{L}(T_{I \times J}) = P_{\textnormal{adm}} \cup P_{\textnormal{non-adm}} .
\]
The \textit{sparsity constant} $c_{\textnormal{sp}}$ (see~\cite{gh03}) is defined as
\[
 c_{\textnormal{sp}} := \max\left\{ \max_{t \in T_{I}} c_{\textnormal{sp}}^{r}(t),\,
 \max_{s \in T_{J}} c_{\textnormal{sp}}^{c}(s)\right\},
\]
where
\[
  c_{\textnormal{sp}}^{r}(t) := \lvert\{ s \subset J : t \times s \in  P\}\rvert,
   \]
denotes the maximum number of blocks $t \times s$ contained in $P$ for a given cluster $t \in T_{I}$ and
 \[
  c_{\textnormal{sp}}^{c}(s) :=\lvert\{ t \subset I  : t \times s \in P\}\rvert
 \]
the maximum number of blocks $t \times s \in P$ for a cluster $s \in T_{J}$. We refer the reader to~\cite{bebendorf08} for more details on the construction of cluster trees.

\subsection{Hierarchical Matrices and Adaptive Cross Approximation}
In view of the construction of the partition $P$, the set of $\mathcal H$-matrices with blockwise rank $k$ is defined by
\begin{equation*}
\mathcal H(P, k) := \{ M \in \R^{I \times J} \, : \, \textnormal{rank } M_b \leq k \textnormal{ for all } b \in P\},
\end{equation*}
see~\cite{hackbusch99, hk00}. A great advantage of hierarchical matrices is the efficient matrix-vector multiplication. The product of an $\mathcal H$-matrix with a vector can be computed in logarithmic-linear time; see~\cite{hackbusch99,hk00,bebendorf08}.

Meanwhile many different methods exist to generate low-rank approximations on admissible matrix blocks. Replacing the kernel function of the integral operator by truncated kernel expansions as it is described in the beginning of Sect.~\ref{eq:approx} of this article
is a common analytical approach. Examples for such expansions are the multipole expansion~\cite{r85, gr87} or interpolating polynomials. Other approaches such as the algebraic pseudo-skeleton method~\cite{gtz97} work directly on the entries of the
considered block. In this article we rely on the adaptive cross
approximation~(ACA) (see~\cite{bebendorf00}) which requires only few of the original entries
to construct the low-rank approximation. Non-admissible blocks cannot be approximated. However, they are small and can be computed entry by entry.

In the following we concentrate on a single admissible block~$A_{ts} \in \R^{t \times s}$ of~$A$.
The following Algorithm~\ref{alg:ACA} (see~\cite{bebendorf00,br03}) constructs two sequences $\{u_k\} \subset \R^{t}$ and
$\{v_k\}\subset \R^s$. The matrix 
\begin{equation*}
S_k := \sum_{l = 1}^k u_l v_l^T
\end{equation*}  
has rank at most~$k$. Given $\eps_{\textnormal{ACA}} > 0$, the remainder $R_k := A_{ts} - S_k$ has relative accuracy 
\begin{equation*}
\norm{R_k}_F \leq \eps_{\textnormal{ACA}} \norm{A_{ts}}_F,
\end{equation*} 
where $\norm{\cdot}_F$ denotes the Frobenius norm, such that $S_k$ can be used as an approximation of $A_{ts}$.
\begin{algorithm}[htb]
\caption{Adaptive Cross Approximation (ACA)}\label{alg:ACA}
\begin{algorithmic}[0]
\State Let $k = 1$; $Z = \emptyset$; $\eps_\textnormal{ACA}> 0$
\Repeat
\State find $i_k$ by some rule
\State $\tilde{v}_k := A_{i_k, s}$
\For {$l = 1,\ldots,k-1$} $\tilde{v}_k := \tilde{v}_k - (u_{l})_{i_k}v_{l}$ \EndFor
\State $Z := Z \cup \{i_k\}$
\If{$\tilde{v}_k$ does not vanish} 
\State $j_k := \textnormal{argmax}_{j\in s} \lvert(\tilde{v}_k)_{j}\rvert$; $v_k := (\tilde{v}_k)_{j_k}^{-1}\tilde{v}_k$ 
\State $u_k := A_{t, j_k}$
\For {$l = 1,\ldots, k-1$} 
$u_k := u_k - (v_{l})_{j_k} u_{l}$\EndFor
\State $k := k + 1$
\EndIf
\Until $\norm{u_{k+1}}_2\norm{v_{k+1}}_2 \leq \frac{\eps_\textnormal{ACA}(1-\beta)}{1+\eps_\textnormal{ACA}} \norm{S_k}_F$ or $Z = t$
\end{algorithmic}
\end{algorithm}

It is easily seen that the vectors $u_k$ and $v_k$ have the representation 
\begin{equation*}
u_k = (R_{k-1})_{t j_k} \quad \textnormal{and} \quad v_k = \frac{1}{(R_{k-1})_{i_k j_k}}  (R_{k-1})_{i_k s}.
\end{equation*}
When selecting the row indices $i_k$, it must be ensured that the Vandermonde matrix corresponding to the system in which the approximation error is to be estimated is not singular; cf.~\cite{bebendorf08}.
In the case of kernel functions of the form $\kappa (x,y) = \xi(x) \zeta(y) \lvert x-y\rvert^{-\alpha}$ with $\alpha > 0$ and $\xi$ and $\zeta$ depending 
on only one of the variables~$x$ and~$y$, respectively, no attention has to be paid to the choice of the
row indices, because in this case a system of functions can be specified which leads to a non-singular
Vandermonde matrix; see~\cite{bbf21}.

The vanishing rows of the remainders $R_k$ are gathered in the set~$Z$. If the $i_k$-th row of $R_k$ is nonzero and therefore used as $v_k$, it is also included in $Z$ as the $i_k$-th row of the next remainder $R_{k+1}$ vanishes.
The number of elements of $Z$ usually depends logarithmically on the desired blockwise precision $\eps_{\textnormal{ACA}}$; see~\cite{bebendorf08}.
In the following, $\lvert Z\rvert$ is to be managed for each block in the case of the matrix-vector multiplication by an adaptive algorithm,
where the quality of the approximation of the respective block of~$A$ is adapted to the structure of the vector~$x$ to be multiplied rather than to the blockwise accuracy~$\eps_{\textnormal{ACA}}$.

\section{The Adaptive Matrix-Vector Multiplication} \label{sec:amvm}
The goal of this section is to introduce an approximate and adaptive algorithm for the multiplication of a matrix $A \in \R^{M \times N}$ by a vector $x \in \R^N$, i.e.
\[
b = Ax,
\]
where $A$ is the discretization of a non-local operator and $b$ denotes the resul-ting vector. Since $A$ is fully populated, the usual way of treating such problems in our case is to approximate the system matrix by hierarchical matrices at first and then to multiply the approximation of~$A$ by the vector~$x$. 
As a result of the construction of the approximation by ACA, redundant and unnecessary information can arise for the simple reason that ACA treats each matrix block independently
such that a prescribed accuracy is guaranteed. 
In order to avoid the generation of such information, we follow an adaptive strategy.
Instead of the previous approach of generating a single hierarchical matrix approximation of~$A$, we construct a sequence of approximations~$A_k$ and the resulting vectors~$b_k:=A_kx$.
The individual approximations are steered using a residual error estimator based on the $h$--$h/2$ strategy~\cite{fp08} and the Dörfler marking technique~\cite{doerfler96}.
Note that in contrast to the conventional field of application of such error estimators, no refinement of the geometry or the grid is considered here.
While the low-rank approximations of the individual blocks are successively improved, the underlying grid structure and the underlying block-cluster tree are not changed at any time.

Of course the above procedure looks much more complex than multiplying a single approximation of~$A$ by the vector~$x$.
The approach here aims to exploit properties of the vector~$x$ in combination with properties of~$A$.
As an example, consider the extreme case that $x=0$. Then the adaptive approach detects that there is no use in
computing an approximation of~$A$ with accuracy~$\eps_\textnormal{ACA}$, while the usual approach would first approximate every block with this accuracy and then perform the multiplication.
Depending on the combination of~$A$ and~$x$, we expect improved memory requirements and computational time.

A reliable estimate of the error requires the existence of a more accurate approximation~$\hat A_k$  of~$A$  than $A_k$, i.e., we assume that the saturation assumption 
\begin{equation} \label{eq:AMVM_satass}
\norm{\hat b_k - b}_2 \leq \csat \norm{b_k - b}_2
\end{equation}
for some $0<\csat< 1$ is fulfilled, where $\hat b_k := \hat A_k x$.  A natural choice (so-called look-ahead approximation) for $\hat A_k$ is the improved approximation that results from $A_k$ by applying a fixed number of additional ACA steps to each admissible 
block and by setting $(\hat A_k)_{ts} = A_{ts}$ for all other non-admissible blocks $t \times s  \in P_{\textnormal{non-adm}}$.
Using the error estimator
\[
\gamma_k := \norm{b_k - \hat b_k}_2 = \norm{\sum_{t\times s \in P} (A_k-\hat A_k)_{ts}x_s}_2,
\]
which is localized with respect to blocks in $P$, the algorithm for the adaptive matrix-vector multiplication is summarized in Algorithm~\ref{alg:AMVM}.
Following the ideas above, we thus combine the assembly of the discretized non-local operator with the simultaneous computation of the matrix-vector multiplication. Figure~\ref{fig:schematic} shows a schematic illustration of the procedure.
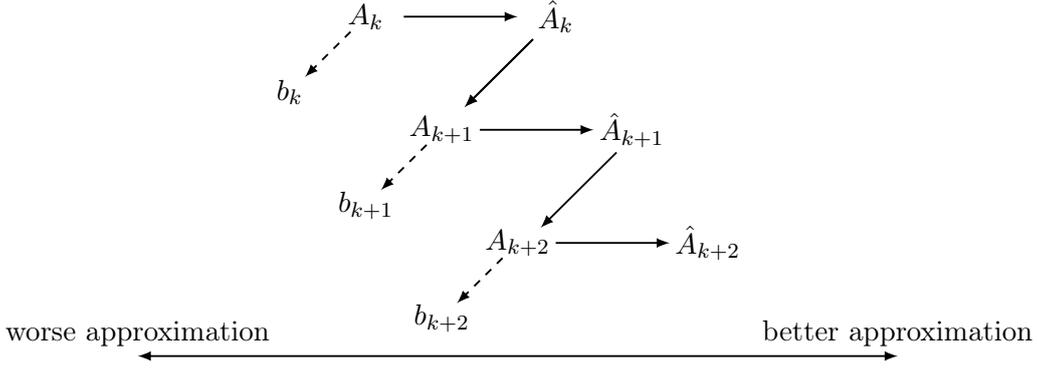
\begin{figure}[htb] \centering 
\begin{tikzpicture}
%\node[rectangle] at (0, 0) {\fcolorbox{red}{red}{$A_{k}$}};
\node[rectangle] at (0, 0) {$A_{k}$};
\node[rectangle] at (-1,-1) {$b_{k}$};
\draw[line width=0.7pt,dashed, ->, >=latex] (-0.2,-0.2)--(-0.8,-0.8);
\draw[line width=0.7pt, ->, >=latex] (0.5,0)--(2.0,0);
\draw[line width=0.7pt, ->, >=latex] (2.2,-0.3)--(1.3,-1.2);
\node[rectangle] at (2.5, 0) {$\hat{A}_{k}$};
\node[rectangle] at (1.0, -1.5) {$A_{k+1}$};
\node[rectangle] at (0.0, -2.5) {$b_{k+1}$};
\draw[line width=0.7pt,dashed, ->, >=latex] (0.8, -1.7)--(0.2,-2.3);
\draw[line width=0.7pt, ->, >=latex] (1.5,-1.5)--(3.0,-1.5);
\node[rectangle] at (3.5, -1.5) {$\hat{A}_{k+1}$};
\draw[line width=0.7pt, ->, >=latex] (2.2,-0.3)--(1.3,-1.2);
\node[rectangle] at (2.0, -3) {$A_{k+2}$};
\node[rectangle] at (1.0, -4) {$b_{k+2}$};
\draw[line width=0.7pt,dashed, ->, >=latex] (1.8, -3.2)--(1.2,-3.8);
\draw[line width=0.7pt, ->, >=latex] (2.5,-3)--(4.0,-3);
\node[rectangle] at (4.5, -3) {$\hat{A}_{k+2}$};
\draw[line width=0.7pt, ->, >=latex] (3.3,-1.8)--(2.3,-2.8);
\draw[line width=0.7pt, <->, >=latex] (-3,-4.5)--(7,-4.5);
\node[rectangle] at (-3.0, -4.2) {worse approximation};
\node[rectangle] at (7.0, -4.2) {better approximation};
%\node[rectangle] at (6.5, -0.5) {The vector $b_{k}$ is the result of the};
%\node[rectangle] at (6.5, -1.0) {multiplication of  $A_{k}$ with $g$,};
%\node[rectangle] at (6.8, -1.5) { i.e. $b_k = A_k g$.};
\end{tikzpicture}
   \caption{Schematic illustration of the procedure.}
   \label{fig:schematic}
   \end{figure}

\begin{algorithm}[htb]
\caption{Adaptive matrix-vector multiplication (AMVM)}\label{alg:AMVM}
\begin{algorithmic}[0]
\State 
\begin{enumerate}
 \item Start with a coarse $\Hm$-matrix approximation $A_{0}$ of $A$ and set $k=0$.
 \item Compute $b_k = A_k x$ and $\hat b_k = \hat A_k x$.
 \item 
 \begin{itemize} 
 \item[a)]Given $0<\theta<1$, find a set of marked blocks $P_k\subset P$ with minimal cardinality such that
 \begin{equation}\label{Dmark}
 \gamma_k-\gamma_k(P_k)\geq \theta\gamma_k,
 \end{equation}
 where $\gamma_k(Q):=\norm{\sum_{t\times s \in P\setminus Q} (A_k-\hat A_k)_{ts}x_s}_2$ and
 $\gamma_k=\gamma_k(\emptyset)=\norm{b_k-\hat b_k}_2$.
 \item[b)] Use the following strategy to construct $P_k$:
 \begin{itemize}
 \item[(i)] Sort the errors $\lvert(b_k - \hat b_k)_i\rvert$, $ i = 1,\ldots,M$, in decreasing order.
 \item[(ii)] Go through the ordered errors step by step starting from the top and detect the corresponding blocks in the considered row $i$.
 \item[(iii)] Add every block $t \times s$ to $P_k$ for which $\lvert[(A_k - \hat A_k)_{ts}x_s]_{i}\rvert \geq (1-\theta)(c_{\textnormal{sp}}MN)^{-1/2} \gamma_k$ holds.
 \item[(iv)] Extend $P_k$ according to (ii) and (iii) as long as condition~\eqref{Dmark} is not fulfilled. 
  \end{itemize}
 \end {itemize}
 
 \item 
 Let
 \[
   A_{k+1}=\begin{cases} (\hat A_k)_b, & b\in P_k,\\
         (A_k)_b, & b\in P\setminus P_k.
   \end{cases}
   \]
  \item If $\gamma_{k+1}>\eps_\textnormal{AMVM}$ increment $k$ and go to 2.
\end{enumerate}
\end{algorithmic}
\end{algorithm}

At first glance Algorithm~\ref{alg:AMVM} uses two $\Hm$-matrices $A_k$ and $\hat A_k$. Since they are strongly related to each other, it is actually sufficient to store only the more accurate approximation $\hat A_k$. 
 Due to the selection criteria of $P_k$ in Algorithm~\ref{alg:AMVM}, clusters of zero entries in the vector $x$ have the consequence that the associated blocks do not have to be approximated at all.
 Hence, this approach allows to take into account the structure of the vector $x$ when approximating~$A$.
 
 \begin{remark}
Notice that the previous algorithm terminates either if in step~3~b)~(iv) condition~\eqref{Dmark} is satisfied
or if the list of blocks has come to its end. In this case also \eqref{Dmark} is valid, because
the condition used in step~3~b)~iii) implies 
$\lvert[(A_k - \hat A_k)_{ts}x_s]_{i}\rvert \leq (1-\theta)(c_{\textnormal{sp}}MN)^{-1/2} \gamma_k$ for all blocks $t\times s\in P \setminus P_k$
and thus
\begin{align*}
\gamma_k(P_k) &= \norm{\sum\limits_{t \times s \in P\setminus P_k} (A_k - \hat A_k)_{ts} x_s}_2 
 = \left( \sum_{i = 1}^M \lvert \left[ \sum_{t \times s \in P\setminus P_k} (A_k - \hat A_k)_{ts} x_s \right]_i \rvert^2\right)^{1/2} \\
& \leq \left( \sum_{i = 1}^M c_{\textnormal{sp}} \sum_{t \times s \in P\setminus P_k} \lvert \left[  (A_k - \hat A_k)_{ts} x_s \right]_i \rvert^2\right)^{1/2} \\
& \leq \left( \sum_{i = 1}^M \sum_{t \times s \in P\setminus P_k,\,i\in t} (1-\theta)^2 (MN)^{-1} \gamma_k^2 \right)^{1/2} \leq (1-\theta) \gamma_k.
\end{align*}

 \end{remark}

The newly introduced algorithm will be examined in more detail in the next steps. First, we consider the reliability and the efficiency as two basic characteristics of the error estimator. 
\begin{lemma}\label{lem:rel}
 Let assumption~\eqref{eq:AMVM_satass} be valid. Then $\gamma_k$ is efficient and reliable, i.e., it holds 
 \[
  c_{\textnormal{eff}} \gamma_k \leq \norm{b_k - b}_2 \leq c_{\textnormal{rel}} \gamma_k,
 \]
 where $c_{\textnormal{eff}} := 1/(1+\csat)$ and $c_{\textnormal{rel}} := 1/(1-\csat)$.
\end{lemma}
\begin{proof}
With the saturation assumption it follows 
\[
\norm{b_k - b}_2 \leq \norm{b_k - \hat b_k}_2 + \norm{\hat b_k + b}_2 \leq \gamma_k + \csat \norm{b_k - b}_2
\]
and thus
\[
\norm{b_k - b}_2 \leq c_{\textnormal{rel}} \gamma_k, 
\]
which proves the reliability of the estimator $\gamma_k$. 
Using again the saturation assumption, we obtain 
\[
\gamma_k = \norm{b_k - \hat b_k}_2 \leq \norm{b_k - b}_2 + \norm{b - \hat b_k}_2 \leq (1+\csat) \norm{b_k - b}_2
\]
and thus
\[
c_{\textnormal{eff}} \gamma_k \leq \norm{b_k - b}_2.
\]
\end{proof}

The next property of the estimator $\gamma_k$ which has to be investigated is the estimator convergence. In order to do this, we must first examine the behavior of the error $ \hat e_k := \norm{\hat b_k - \hat b_{k+1}}_2$, where $\hat b_k = \hat A_k x$.
\begin{lemma}\label{lem:convHatError}
  The error $\hat e_k $ converges to zero for $k \rightarrow \infty$.
\end{lemma}
\begin{proof}
 For $\hat e_k$ it holds
 \begin{align*}
  \hat e_k = \norm{(\hat A_k  - \hat A_{k+1}) x}_2 &\leq \norm{\hat A_k - \hat A_{k+1}}_2 \norm{x}_2 
  = \norm{\hat{R}_k - \hat{R}_{k+1}}_2 \norm{x}_2 \\
  &\leq \left( \norm{\hat{R}_k}_2 + \norm{\hat{R}_{k+1}}_2\right) \norm{x}_2,
 \end{align*}
 where $\hat R_k:=A-\hat A_k$ and $\hat R_{k+1}:=A-\hat A_{k+1}$ denote the remainders of the whole matrix at this point. Since the remainders $\norm{\hat{R}_k}$ and $\norm{\hat{R}_{k+1}}$ converge to zero for $k \rightarrow \infty$, the error $\hat e_k$ converges to zero for $k \rightarrow \infty$.
\end{proof}
The convergence of the error estimator $\gamma_k$ can be proven via an estimator reduction principle. 
\begin{lemma}[estimator reduction]\label{lem:estred}
Let $s > 1$ and $1 - \frac{1}{\sqrt{s}} < \theta < 1$ be given. Then it holds that 
\[
\gamma_{k+1}^2 \leq c_1 \gamma_{k}^2 + c_2 \hat e_k^2,
\]
where $c_1  = 1/s < 1$ and $c_2 = [1-s(1-\theta)^2]^{-1}$. Furthermore, 
$\lim_{k \to \infty} \gamma_k = 0$.
\end{lemma}
\begin{proof}
We have a closer look at the error estimator $\gamma_{k+1}$. With $\delta > 0$ and Young's inequality it follows
  \begin{align*}
  \gamma_{k+1}^2 &= \norm{b_{k+1} - \hat b_{k+1}}_2^2=  \norm{b_{k+1} - \hat b_k + \hat b_k - \hat b_{k+1}}_2^2 \\
  &\leq \left( \norm{b_{k+1}-\hat b_k}_2 + \norm{\hat b_k - \hat b_{k+1}}_2\right)^2\\
  &\leq (1+\delta) \underbrace{\norm{b_{k+1} - \hat b_{k}}^2_2}_{=: e^2} + (1+1/\delta) \hat e_k^2.
 \end{align*}
 If we split up $e$ into the marked and non-marked blocks, from \eqref{Dmark} we obtain the following estimator reduction:
 \begin{align*}
  e = \norm{b_{k+1} - \hat b_k}_2 &= \norm{\sum_{t\times s\in P} (A_{k+1}-\hat A_k)_{ts}x_s}_2 \\
  &= \norm{\sum_{t\times s\in P_k} (A_{k+1}-\hat A_k)_{ts}x_s+\sum_{t\times s\in P\setminus P_k} (A_{k+1}-\hat A_k)_{ts}x_s}_2\\
  &= \norm{\sum_{t\times s\in P\setminus P_k} (A_k-\hat A_k)_{ts}x_s}_2=\gamma_k(P_k)\leq (1-\theta)\gamma_k.
 \end{align*}
With the choice $\delta = \frac{1- s(1-\theta)^2}{s(1-\theta)^2}$ we get
\begin{align*}
 \gamma_{k+1}^2 &\leq (1+\delta) (1-\theta)^2 \gamma_k^2 + (1+1/\delta) \hat e_k^2 \\
 & = \left( 1 + \frac{1 -s(1-\theta)^2}{s(1-\theta)^2} \right) (1-\theta)^2 \gamma_k^2 + \left( 1 + \frac{s(1-\theta)^2}{1 - s(1-\theta)^2}\right) \hat e_k^2 \\
 & = \frac{1}{s} \gamma_k^2 + \frac{1}{1-s(1-\theta)^2} \hat e_k^2.
\end{align*}

The second part of the assertion is proved with Lemma~\ref{lem:convHatError} and the error estimator reduction principle introduced in~\cite{afl12}. Let $\hat E > 0$ be a number satisfying $\hat e_k \leq \hat E$ for all $k$. The estimator reduction 
principle leads to
\begin{align*}
\gamma_{k+1}^2 &\leq c_1 \gamma_k^2 + c_2 \hat e_k^2 \leq c_1 ( c_1 \gamma_{k-1}^2 + c_2 \hat e_{k-1}^2) + c_2 \hat e_k^2\\
& \leq \ldots \leq c_1^{k+1} \gamma_0^2 + c_2 \sum_{i = 0}^k c_1^{k-i} \hat e_k^2 \\
& \leq c_1^{k+1} \gamma_0^2 + c_2 \hat E \sum_{l = 0}^k c_1^l \leq \gamma_0^2 + \frac{c_2 \hat E}{1-c_1}.
\end{align*}
Accordingly, the sequence $\{\gamma_k\}_{k \in \N_0}$ is bounded and we are able to define $\Gamma := \limsup_{k \rightarrow \infty} \gamma_k^2$. Using the estimator reduction principle once more yields
\[
\Gamma = \limsup_{k \to\infty} \gamma_{k+1}^2 \leq c_1 \limsup_{k \to \infty} \gamma_k^2 + c_2 \underbrace{\limsup_{k \rightarrow \infty} \hat e_k^2}_{= 0} = c_1 \Gamma.
\]
Thus $\Gamma = 0$ and it follows
\[
0 \leq \liminf_{k \to \infty} \gamma_k \leq \limsup_{k \to \infty} \gamma_k = \Gamma=0,
\]
which shows
\[
\lim_{k \rightarrow \infty} \gamma_k = 0.
\]
\end{proof}
Exploiting the reliability of the error estimator, the convergence of the adaptive matrix-vector multiplication can also be shown.
\begin{lemma}[estimator convergence] Let the requirements of Lemma~\ref{lem:convHatError} and Lemma~\ref{lem:estred} be valid. Then, the error $\norm{b_k - b}_2$ of the sequence $\{b_k\}_{k \in \N}$ constructed by Algorithm~\ref{alg:AMVM}
 converges to zero.
\end{lemma}
\begin{proof}
Using the reliability of the estimator and the reduction principle leads to 
\[
\norm{b_k - b}_2 \leq c_{\textnormal{rel}} \gamma_k \rightarrow 0 \quad \text{for }k \to \infty.
\]
\end{proof}

\section{Boundary Integral Approximation of Linear Elasticity} \label{sec:BI_LE}
\subsection{Integral Formulation of Linear Elasticity}\label{subsec:four_one}
We assume that $\Omega \subset \R^3$ is a Lipschitz domain and its boundary $\partial\Omega=\Gamma_D\cup\Gamma_N$ is
partitioned into a Dirichlet boundary~$\Gamma_D$ and a Neumann boundary~$\Gamma_N$. The solution of the equations of linear elasticity can be written as
\[
u(x) = \tilde V \gamma_1^\textnormal{int} u - \tilde W \gamma_0^\textnormal{int} u,
\]
where 
\[
\tilde V f(x) := \int_{\partial\Omega} S(x,y) f(y) \ud s_y, \quad f \in [H^{-1/2}(\partial\Omega)]^3,
\]
denotes the single-layer potential and 
\[
\tilde W g(x) := \int_{\partial\Omega} \gamma_{1,y}^\textnormal{int} S(x,y) g(y) \ud s_y, \quad g \in [H^{1/2}(\partial\Omega)]^3,
\]
denotes the double-layer potential. The fundamental solution $S(x,y) := S_K(x-y)$ of linear elasticity is given by Kelvin's solution tensor
\[
S_K(x) := \left( \frac{1}{8\pi} \frac{1}{E} \frac{1+\nu}{1-\nu} \left[ \frac{3-4\nu}{\lvert x \rvert} \delta_{ij} + \frac{x_i x_j}{\lvert x \rvert^3}\right] \right)_{ij}\in\R^{3\times 3}
\] 
for $x \in \R^3 \setminus\{0\}$.
Using the trace operators $\gamma_0^\textnormal{int}$ and $\gamma_1^\textnormal{int}$, we define the single-layer operator $V := \gamma_0^\textnormal{int} \tilde V$ and
the hyper-singular operator $D := -\gamma_1^\textnormal{int}\tilde W$. Furthermore, we are going to use the
the double-layer operator 
\[
(Kg) (x) = \lim_{\varepsilon \rightarrow 0} \int_{\partial\Omega\setminus B_{\varepsilon}(x)} \gamma_{1,y}^\textnormal{int} S(x,y) g(y) \ud s_y, \quad x \in \partial\Omega, \quad g \in [H^{1/2}(\partial\Omega)]^3
\]
and its adjoint
\[
(K'f)(x) = \lim_{\varepsilon \rightarrow 0} \int_{\partial\Omega \setminus B_{\varepsilon}(x)} \gamma_{1,x}^\textnormal{int} S(x,y) f(y) \ud s_y, \quad x \in \partial\Omega, \quad f \in [H^{-1/2}(\partial\Omega)]^3.
\]
In order to solve  mixed boundary value problems (cf.~Sect.~\ref{sec:one}), boundary integral operators have to be defined on the respective part of the boundary.
On the Dirichlet boundary~$\Gamma_D$ we set
\[
V_{DD}: [\tilde H^{-1/2}(\Gamma_D)]^3 \rightarrow [H^{1/2}(\Gamma_D)]^3, \quad V_{DD} f:= (V \tilde f)\rvert_{\Gamma_D},
\]
where $\tilde H^{-1/2}(\Gamma_D) = [H^{1/2}(\Gamma_D)]'$
and $f = \tilde f\rvert_{\Gamma_D}$ with $\tilde f \in [H^{-1/2}(\partial\Omega)]^3$ and $\textnormal{supp } \tilde f \subset \Gamma_D$. Using the extension $\tilde g \in [H^{1/2}(\partial\Omega)]^3$
 of a function $g \in [\tilde H^{1/2}(\Gamma_N)]^3$, where \[\tilde H^{1/2}(\Gamma_N) := \{ v = \tilde v\rvert_{\Gamma_N}:\tilde v \in H^{1/2}(\partial\Omega), \, \textnormal{supp } \tilde v \subset \Gamma_N\},\]
 we define
\[
D_{NN}: [\tilde H^{1/2}(\Gamma_N)]^3 \rightarrow [H^{-1/2}(\Gamma_N)]^3, \quad D_{NN} g := (D \tilde g)\rvert_{\Gamma_N}
\]
with $ H^{-1/2}(\Gamma_N) = [\tilde H^{1/2}(\Gamma_N)]'$.
The following two operators describe the interaction between the Dirichlet and the Neumann data. We define the double-layer operator of the Neumann boundary 
\[
K_{ND}: [\tilde H^{1/2}(\Gamma_N)]^3 \rightarrow [H^{1/2}(\Gamma_D)]^3, \quad K_{ND} g := (K \tilde g)\rvert_{\Gamma_D}
\]
and the adjoint double-layer operator of the Dirichlet boundary
\[
K'_{DN}: [\tilde H^{-1/2}(\Gamma_D)]^3 \rightarrow [H^{-1/2}(\Gamma_N)]^3, \quad K'_{DN} f := (K'\tilde f)\rvert_{\Gamma_N}.
\]
Then, the boundary value problem 
\begin{align*}
-\sum_{j = 1}^3 \frac{\partial}{\partial x_j} \sigma_{ij}(u,x) &= 0, \quad x \in \Omega, \\
u(x) &= g_D(x), \quad x \in \Gamma_D, \\
\sum_{j = 1}^3 \frac{\partial}{\partial x_j} \sigma_{ij}(u,x) n_j(x) &= g_N(x), \quad x \in \Gamma_N
\end{align*}
has the solution 
\begin{equation} \label{eq:solution}
v = \tilde V(\tilde g_N + \tilde t) - \tilde W (\tilde g_D + \tilde u),
\end{equation}
where $\tilde u \in [H^{1/2}(\partial\Omega)]^3$ and $\tilde t \in [H^{-1/2}(\partial\Omega)]^3$ denote the extensions by zero of the functions $u \in [\tilde H^{1/2}(\Gamma_N)]^3$ and $t \in [\tilde H^{-1/2}(\Gamma_D)]^3$, which are the solutions
of the integral equations
\begin{align*}
V_{DD} t - K_{ND} u &= \left( \frac{1}{2} I + K \right) \tilde g_{D}\rvert_{\Gamma_D} - V \tilde g_N\rvert_{\Gamma_D}, \\
K'_{DN} t + D_{NN} u &= -D \tilde g_D\rvert_{\Gamma_N} + \left( \frac{1}{2} I - K' \right) \tilde g_N\rvert_{\Gamma_N}.
\end{align*}
In~\eqref{eq:solution} the given Dirichlet data $g_D \in [H^{1/2}(\Gamma_D)]^3$ and Neumann data $g_N \in [H^{-1/2}(\Gamma_N)]^3$ are extended to the functions $\tilde g_D \in [H^{1/2}(\partial\Omega)]^3$ and 
$\tilde g_N \in [H^{-1/2}(\partial\Omega)]^3$; see~\cite{os08}.

A stable numerical treatment of the above operators is only possible if the singularities are not too strong. Since $V$ is weakly singular, we do not expect any numerical problems. For $D$ and $K$ weakly singular representations have to
be found. Using the boundary differential operators
\[
T_{ij}(x) := n_j(x) \frac{\partial}{\partial x_j} - n_i(x) \frac{\partial}{\partial x_i}, \quad i,j = 1,2,3, \quad x \in \partial\Omega,
\]
and 
\[
\frac{\partial}{\partial S_1} (x) := T_{32}(x), \quad  \frac{\partial}{\partial S_2}(x) := T_{13}(x), \quad \frac{\partial}{\partial S_3} (x) := T_{12}(x),
\]
the double-layer operator $K$ can be rewritten as 
\[
Ku(x) = \frac{1}{4\pi} \int_{\partial\Omega} \frac{\partial}{\partial n_y} \frac{1}{\lvert x-y \rvert} u(y) \ud s_y - \frac{1}{4\pi} \int_{\partial\Omega} \frac{1}{\lvert x-y \rvert} T u(y) \ud s_y + 2 \mu V T u(x)
\]
for $u \in [H^{1/2}(\partial\Omega)]^3$. The operator $D$ in terms of weakly singular integrals has the representation
\begin{align*}
\langle Du, v \rangle_{\partial\Omega} &= \frac{\mu}{4\pi} \int_{\partial\Omega} \int_{\partial\Omega} \frac{1}{\lvert x-y \rvert} \left( \sum_{k = 1}^3 \frac{\partial}{\partial S_k} u(y) \cdot \frac{\partial}{\partial S_k} v(x) \right) \ud s_x \ud s_y \\
& \, + \frac{\mu}{2\pi} \int_{\partial\Omega} \int_{\partial\Omega} (T(x) v(x))^T \frac{I}{\lvert x-y \rvert} (T(y) u(y))^T \ud s_x \ud s_y - 4\mu^2 \langle VT u, Tv \rangle_{\partial\Omega} \\
& \, + \frac{\mu}{4\pi} \int_{\partial\Omega} \int_{\partial\Omega} \sum_{i,j,k = 1}^3 T_{kj}(x)v_i(x) \frac{1}{\lvert x-y \rvert} T_{ki}(y)) v_j(y) \ud s_x \ud s_y,
\end{align*}
see~\cite{han94}.

\subsection{Discretization Techniques} \label{sec:discretization}
Our goal is the computation of a numerical solution of the integral equations for linear elasticity via the boundary element method~(BEM). The starting point is an admissible triangulation
$\mathcal T_h$ of the surface of the computational domain $\Omega$ in regular triangles $\tau_i$, $i = 1, \dots, M$, and nodes $p_j$, $j = 1,\dots,N$. Here a triangulation is called admissible, if neighboring triangles have only one common 
edge or node; see~\cite{ss11}.

On the triangulation $\mathcal T_h$ the space of piecewise linear functions $\mathcal S^0_h$ is given by its basis
\[
\varphi_i (x) = \begin{cases} 1,& \quad x \in \tau_i, \\ 0,& \quad \textnormal{else},\end{cases} \quad i = 1,\dots,M,
\]
which is used for the discretization of the operator $V$.  The operator $K$ and parts of $D$ are discretized using functions
\[
\psi_j(x) := \begin{cases} 1,& \quad x = p_j, \\ 0,& \quad x = p_l \neq p_j, \\ \textnormal{linear},& \quad \textnormal{else}, \end{cases}
 \quad j  = 1,\dots,N,
\]
defining a basis of the space $\mathcal S_h^1(\Gamma)$ of continuous and piecewise linear functions.
We find solutions of the form 
\[
t_h(x) = \sum_{i = 1}^M \begin{bmatrix} t_i^{(1)} \\ t_i^{(2)} \\ t_i^{(3)} \end{bmatrix} \varphi_i(x) \quad \textnormal{and} 
\quad u_h(x) = \sum_{j = 1}^N \begin{bmatrix} u_j^{(1)} \\ u_j^{(2)} \\ u_j^{(3)} \end{bmatrix}\psi_j(x).
\]
The coefficient vectors  $t = [ t_i^{(1)}, t_i^{(2)}, t_i^{(3)}]_{i = 1}^M$ and $u = [ u_j^{(1)}, u_j^{(2)}, u_j^{(3)}]_{j = 1}^N$ are the solution of the linear system of equations
\begin{equation}\label{eq:sysline}
\begin{bmatrix} V_{DD,h} & -K_{ND,h} \\ K_{ND,h}^T & D_{NN,h} \end{bmatrix} \begin{bmatrix} t \\ u \end{bmatrix}
= \begin{bmatrix} -V & \frac{1}{2} M  + K   \\ \frac{1}{2} M - K^T & -D \end{bmatrix} \begin{bmatrix} \tilde g_N \\ \tilde g_D \end{bmatrix}
=: \begin{bmatrix} f_D \\ f_N \end{bmatrix}
\end{equation}
\begin{align*}
V_{DD,h}[ij] &= \langle V_{DD} \varphi_j , \varphi_i \rangle_{\Gamma_D}, \quad K_{ND,h}[jk] = \langle K_{ND} \psi_k, \varphi_j \rangle_{\Gamma_D}, \\
D_{NN,h}[kl] &= \langle D_{NN} \psi_l, \psi_k \rangle_{\Gamma_N}
\end{align*}
for $i,j = 1,\dots,M$ and $k,l = 1,\dots,N$. 

At the end of this section we want to state representations for the operators $V_{DD,h}$, $K_{ND,h}$, and~$D_{NN,h}$, which are more advantageous for numerical calculations. Using Kelvin's solution tensor and 
a suitable space for the discretization of the operators like the space $[\mathcal{S}^0(\Gamma)]^3$, the stiffness matrix $V_h \in \R^{3M \times 3M}$ of the single-layer potential has the representation
\begin{equation}\label{eq:Vh}
V_h = \frac{1}{2} \frac{1}{E} \frac{1+\nu}{1-\nu} \left( (3-4\nu) \begin{bmatrix} V_{\Delta, h} & 0 & 0 \\ 0 & V_{\Delta, h} & 0 \\ 0 & 0 & V_{\Delta, h} \end{bmatrix} +  
\begin{bmatrix} V_{11} & V_{12} & V_{13} \\ V_{12} & V_{22} & V_{23} \\ V_{13} & V_{23} & V_{33} \end{bmatrix} \right),
\end{equation}
where 
\[
V_{\Delta,h, ij} = \frac{1}{4} \int_{\tau_j} \int_{\tau_i} \frac{1}{\lvert x-y \rvert} \ud s_y \ud s_x
\]
and 
\[
V_{kl}[ij] = \frac{1}{4} \int_{\tau_j} \int_{\tau_i} \frac{(x_k - y_k)(x_l - y_l)}{\lvert x-y \rvert^3} \ud s_y \ud s_x
\]
are $M \times M$ sub-matrices for $k,l = 1,2,3$.  Together with the space $[\mathcal S^1(\Gamma)]^3$, the operator $K_h$ can be represented as
\begin{equation}\label{eq:Kh}
K_h = \begin{bmatrix} K_{\Delta , h} & 0 & 0 \\ 0 & K_{\Delta , h} & 0 \\ 0 & 0 & K_{\Delta , h} \end{bmatrix} - \begin{bmatrix} V_{\Delta, h} & 0 & 0 \\ 0 & V_{\Delta, h} & 0 \\ 0 & 0 & V_{\Delta, h} \end{bmatrix} 
T_h + \frac{E}{1+\nu} V_h T_h,
\end{equation}
where 
\[ 
K_{\Delta, h}[ij] = \frac{1}{4\pi} \sum_{\tau \, \in \, \textnormal{supp }\psi_j} \int_{\tau} \int_{\tau_i} \frac{(x-y)^T n(y)}{\lvert x-y \rvert^3} \psi_j(y) \ud s_y \ud s_x, 
\]
$i = 1,\dots,M, \; j = 1,\dots,N,$ and
\[ 
T_h := \begin{bmatrix} 0 & T_{12,h} & T_{13,h} \\ -T_{12,h} & 0 & T_{23,h} \\ -T_{13,h} & -T_{23,h} & 0 \end{bmatrix}, \quad T_{kl,h}[ij] := T_{kl}(x) \psi_j(\hat x), \, \hat x \in \tau_i, 
\] 
for $k,l \in \{1,2,3\}$, $i=1,\dots,M$, $j=1,\dots,N$. Finally, the matrix $D_h$ is given by
\begin{equation}\label{eq:Dh}
\begin{split}
D_h &= \sum_{k = 1}^3 \frac{\mu}{4\pi} S_{k,h}^T  \begin{bmatrix} V_{\Delta, h} & 0 & 0 \\ 0 & V_{\Delta, h} & 0 \\ 0 & 0 & V_{\Delta, h} \end{bmatrix} S_{k,h}
+ \frac{\mu}{2\pi}  T_h^T \begin{bmatrix} V_{\Delta, h} & 0 & 0 \\ 0 & V_{\Delta, h} & 0 \\ 0 & 0 & V_{\Delta, h} \end{bmatrix} T_h \\
 & \quad + 4 \mu^2 T_h^T V_h T_h + \frac{\mu}{4\pi} D'_h
\end{split}
\end{equation}
with 
\[
D'_h := \begin{bmatrix} D'_{11,h} & D'_{12,h} & D'_{13,h} \\ D'_{21,h} & D'_{22,h} & D'_{23,h} \\
  D'_{31,h} & D'_{32,h} & D'_{33,h} \end{bmatrix}, 
\quad D'_{ij,h} := \sum_{k = 1}^3 T_{kj,h}^T V_{\Delta, h} T_{ki,h}
\]
and 
\begin{align*}
S_{1,h} &:= \begin{bmatrix} T_{32,h} & 0 & 0 \\ 0 & T_{32,h} & 0 \\ 0 & 0 & T_{32,h} \end{bmatrix}, \quad  S_{2,h} := \begin{bmatrix} T_{13,h} & 0 & 0 \\ 0 & T_{13,h} & 0 \\ 0 & 0 & T_{13,h} \end{bmatrix}, \\
S_{3,h} &:=  \begin{bmatrix} T_{21,h} & 0 & 0 \\ 0 & T_{21,h} & 0 \\ 0 & 0 & T_{21,h} \end{bmatrix},
\end{align*}
see~\cite{rs07}. Using specific restriction operators defined in~\cite{rs07} allows the re-presentation of the discretized operators~$V_h$, $K_h$, and~$D_h$ with respect to the corresponding boundaries, resulting in the operators~$V_{DD,h}$, $K_{ND,h}$, and~$D_{NN,h}$

\section{The Adaptive Solution of Lam\'e Equations}\label{sec:adapt_BACA}
The adaptive matrix-vector multiplication introduced in Sect.~\ref{sec:amvm} can be applied in the
context of boundary element methods when the given data vectors are multiplied by discrete integral
operators on the the right-hand side of the discretized integral equation.
If also the solution of the latter is to be computed, then the BACA method introduced in~\cite{bb21}
can be employed. It combines the adaptive construction of the $\mathcal{H}$-matrix approximation of the system matrix
with the simultaneous iterative solution of the system.
The individual blocks are approximated only as accurate as necessary for the prescribed accuracy and
the given right-hand side vector. 
The BACA was developed for the Laplace equation and is thus not adapted to the structure of the Lamé equations. In this section, BACA will be modified to allow its application
to problems from linear elasticity.

We consider the numerical solution of the linear system $Ax=f$ from \eqref{eq:sysline} 
with \[A = \begin{bmatrix} V_{DD,h} & -K_{ND,h} \\ K_{ND,h}^T & D_{NN,h} \end{bmatrix}, \quad f = \begin{bmatrix} f_D \\ f_N \end{bmatrix}, \quad \textnormal{ and } \quad x = \begin{bmatrix} t \\ u \end{bmatrix}.\]
Each of the four sub-matrices of $A$ consists again of nine sub-matrices with an associated block-cluster tree. Let $P$ be the union of all admissible partitions concerning the boundaries and the different operators, i.e.,
\begin{equation*}
P := P_{V} \cup P_{K} \cup P_{D},
\end{equation*}
where $P_{V}$, $P_{K}$, and $P_{D}$ consist of all blocks of the discretized single-layer operator $V_{h}$, the discretized double-layer operator $K_{h}$ and the discretized hyper-singular operator $D_{h}$
and denote by $P_\textnormal{adm}$ the admissible blocks contained in~$P$. Accordingly, $c_{\textnormal{sp},V}$, $c_{\textnormal{sp},K}$, and $c_{\textnormal{sp},D}$ are the sparsity constants associated with the block-cluster trees for the operators~$V_h$, $K_h$, and~$D_h$.   
The constructed matrix approximation is denoted by 
\[
A_k = \begin{bmatrix} V_{DD,k} & -K_{ND,k} \\ K_{ND,k}^T & D_{NN,k} \end{bmatrix},
\]
where the sub-matrices $V_{DD,h}$, $K_{ND,h}$, and $D_{NN,h}$ are approximated individually. Moreover, let 
\[
\hat A_k = \begin{bmatrix} \hat V_{DD,k} & -\hat K_{ND,k} \\ \hat K_{ND,k}^T & \hat D_{NN,k} \end{bmatrix} 
\]
be a more accurate approximation of~$A$ than~$A_k$. We assume that the saturation condition 
\begin{equation} \label{eq:satElas}
\norm{\hat A_k x_k - A x_k}_2  \leq \csat \norm{A_k x_k - A x_k}_2
\end{equation}
is fulfilled for some $0 < \csat < 1$, where $x_k $ denotes the solution of the linear system $A_k x_k = b$.
Again, a possible strategy for choosing $\hat A_k$ is to add a fixed number of additional ACA steps for each admissible block of~$A_k$ (look-ahead approximation)
and to set $(\hat A_k)_{ts} = A_{ts}$ for all other blocks~$t \times s \in P_{\textnormal{non-adm}}$.

In order obtain some information about the error of the approximation, we use the error estimator
\[
\mathcal E_k^2 := \sum_{t \times s \in P} \norm{(A_k - \hat A_k)_{ts} (x_k)_s}_2^2.
\]
If BACA is to be applied to the saddle-point problem~\eqref{eq:sysline}, we first have to employ
the Bramble-Pasciak conjugate gradient method~\cite{bp88} as the iterative solver. The difference to the case of the Laplace equation lies
in the selection of the blocks to be refined. We use the following strategy based on the
the representations \eqref{eq:Vh}, \eqref{eq:Kh}, and \eqref{eq:Dh} of the discretized operators~$V_h$, $K_h$, and $D_h$,
where we leave the fixed matrices~$T_h$, $S_{1,h}$, $S_{2,h}$ and $S_{3,h}$ unchanged during the whole procedure. Since the sub-matrix $V_{\Delta,h}$
is contained in all the operators~$V_h$, $K_h$, and $D_h$, the refinement of $V_{\Delta,h}$ is implemented at first.
Afterwards the refinements of $V_{ij}$, $i,j = 1,2,3$, $K_{\Delta,h}$, $D'_{h}$ and $D_{h}$ follow. The approximation of a block
is improved only if it has been selected. This leads to the following Algorithm~\ref{alg:BACA_LE}.
\begin{algorithm}[H]
\caption{Block-adaptive ACA for linear elasticity}\label{alg:BACA_LE}
\begin{algorithmic}[0]
\State 
\begin{enumerate}
 \item Start with a coarse $\Hm$-matrix approximation $A_{0}$ of $A$ and set $k=0$.
 \item Given $\alpha\geq 0$, apply the Bramble-Pasciak-CG to the linear system $A_k x_k=b$ until the residual error satisfies
 \begin{equation} \label{eq:solvCond}
 \norm{b-A_kx_k}_2\leq \alpha\norm{(A_k-\hat A_k)x_k}_2
 \end{equation}
 (use $x_{k-1}$ as a starting vector; $x_{-1}:=0$).
 \item Given $0<\theta<1$, find a set of marked blocks $M_k\subset P_\textnormal{adm}$ with minimal cardinality such that
 \begin{equation}\label{Dmark1}
 \mathcal E_k(M_k)\geq \theta\,\mathcal E_k,
 \end{equation}
 where $\mathcal E_k^2(M):=\sum_{t\times s\in M} \norm{(A_k-\hat A_k)_{ts}(x_k)_s}_2^2$ and
 $\mathcal E_k:=\mathcal E_k(P_\textnormal{adm})$.
 \item 
 
Consider refinement in the following order: 
\begin{itemize}
\item[(i)] If blocks in $D_{NN}$ are selected, set
\begin{equation*}
D_{NN, k+1} = \hat D_{NN, k}.
\end{equation*}
\item[(ii)] If blocks in $K_{ND}$ are selected, set 
\begin{equation*}
V_{DD,k+1} = \hat V _{DD, k}  
\end{equation*}
and $(K_{\Delta, h, k+1})_b = (\hat K_{\Delta, h, k})_b$ for all selected blocks $b$ associated with the operator $K_{\Delta,h}$. 
\item[(iii)] If only blocks in $V_{DD}$ are selected, set $(V_{\Delta, h, k+1})_b = (\hat V_{\Delta, h, k})_b$ or $(V_{ij})_b = (\hat V_{ij})_b$, $i = 1,2,3$, for all respective selected blocks $b$.
\end{itemize}
All blocks not selected remain at the current stage of approximation.
 \item If $\mathcal E_{k+1}>\eps_\textnormal{BACA}$ increment $k$ and go to 2.
\end{enumerate}
\end{algorithmic}
\end{algorithm}

In the following we adapt the convergence analysis (presented in~\cite{bb21}) to the previous
method. For the efficiency of the error estimator or at least a lower bound on the expression $\norm{b - A x_k}_2$, we refer to~\cite{bb21}.
The reliability of the estimator follows from the saturation assumption.
\begin{lemma}\label{lem:EstRel}
Let the saturation assumption~\eqref{eq:satElas} be valid. Then $\mathcal E_k$ is reliable, i.e.\ it holds
\[
\norm{b - A x_k}_2 \leq \frac{1 + \alpha (1+\csat)}{1-\csat} \norm{(A_k - \hat A_k) x_k}_2 \leq  \sqrt{27  C_{\textnormal{sp}} \mathcal L} \,\frac{1 + \alpha (1+\csat)}{1-\csat} \mathcal E_k,
\]
where $L$ is the maximum depth of the used cluster trees and $C_{\textnormal{sp}} := \max \{ c_{\textnormal{sp}V}, c_{\textnormal{sp},K}, c_{\textnormal{sp},D}\}$.
\end{lemma}
\begin{proof}
The first assertion follows with condition~\eqref{eq:solvCond} and the saturation assumption from 
\begin{align*}
\norm{b - A x_k}_2 &\leq \norm{b - A_k x_k}_2 + \norm{(A_k- \hat A_k) x_k}_2 + \norm{\hat A_k x_k - A x_k}_2\\
& \leq (\alpha +1) \norm{(A_k - \hat A_k) x_k}_2 + \csat \norm{A_k x_k - A x_k}_2 \\
& \leq (\alpha +1 + \csat \alpha)  \norm{(A_k - \hat A_k) x_k}_2 + \csat \norm{b - A x_k}_2.
\end{align*}
The second inequality is a result of the decomposition of the sub-matrices of $A=\sum_{l=1}^{L} A^{(l)}$ into a sum of level matrices~$A^{(l)}$. Due to the fact that there are several cluster trees involved and the maximal $L$, no more further sub-matrices will exist at a certain level. In this case, use the 
zero sub-matrix for the remaining levels.  We observe
\begin{align*}
\norm{(A_k - \hat A_k) x_k}_2^2 & \leq \left( \sum_{l = 1}^{L} \norm{(A_k - \hat A_k)^{(l)} x_k}_2 \right)^2 \leq L \sum_{l = 1}^{L} \norm{(A_k - \hat A_k)^{(l)} x_k}_2^2 \\
&= L \sum_{l = 1}^{ L} \sum_{t \in T_I^{(l)}} \norm{ \sum_{s : t \times s \in P} (A_k - \hat A_k)_{ts} (x_k)_s}_2^2 \\
& \leq L \sum_{l = 1}^{L} \sum_{t \in T_I^{(l)}} \left(  \sum_{s : t \times s \in P} \norm{ (A_k - \hat A_k)_{ts} (x_k)_s}_2 \right)^2 \\
& \leq 27 C_{\textnormal{sp}} L \sum_{l = 1}^{L}  \sum_{t \in T_I^{(l)}} \sum_{s : t \times s \in P} \norm{ (A_k - \hat A_k)_{ts} (x_k)_s }_2^2 \\
&= 27 C_{\textnormal{sp}} L \sum_{t \times s \in P} \norm{ (A_k - \hat A_k)_{ts} (x_k)_s}_2^2 \\
&= 27 C_{\textnormal{sp}} L \mathcal E_k^2,
\end{align*}
since each of the three disrcretized operators consists of nine sub-operators.
\end{proof}
Except for the inclusion of different sparsity constants and tree depths, there are no other differences in the convergence proof of the adapted BACA method compared to BACA for the Laplace equation. 
For this reason, we refer to the proofs in~\cite{bb21} for the rest of the convergence analysis. 

After having calculated the boundary data $(t_h,u_h)$, i.e.
\[
t_h = \sum_{i = 1}^M t_i \varphi_i \quad \textnormal{and} \quad u_h = \sum_{j = 1}^N u_j \psi_j
\] 
with coefficient vectors $t,u\in\R^3$, the solution $u_h$ in $\Omega$ can be evaluated by
\[
u_h(x) = \sum_{j = 1}^M t_j \int_{\partial\Omega} S(x,y) \varphi_j(y) \ud s_y - \sum_{k = 1}^N u_k \int_{\partial\Omega} \gamma_{1,y}^\textnormal{int} S(x,y) \psi_k(y) \ud s_y,
\]
for $x \in \Omega$, and the stresses $\sigma(u_h, x)$ can be computed using the derivatives  
\begin{equation}\label{eq:partiu}
\partial_{x_i} u_h(x) = \sum_{j = 1}^M t_j \int_{\partial\Omega} \partial_{x_i} S(x,y) \varphi_j (y) \ud s_y - \sum_{k = 1}^N u_k \int_{\partial\Omega} \partial_{x_i} (\gamma_{1,y}^\textnormal{int} S(x,y)) \psi_k(y) \ud s_y, 
\end{equation}
for $i = 1,2,3$ together with Hook's law. If, for instance, the deformations are to be analyzed at several points~$x_1, \dots, x_l$, $l \in \N$, this can be understood as the computation of a vector 
\begin{equation} \label{eq:vector_eq}
v:= \tilde V_h t - \tilde W_h u
\end{equation}
with $v = [u_h(x_i)]_{i=1,\dots,l}$. 
Since the two discrete operators
\[\tilde V_h:=\left[\int_{\partial\Omega}  S(x_i,y) \varphi_j (y) \ud s_y\right]_{ij}
\quad\text{and}\quad 
\tilde W_h:=\left[\int_{\partial\Omega} \gamma_{1,y}^\textnormal{int} S(x_i,y) \psi_k(y) \ud s_y\right]_{ik},
\] 
with $i = 1,\dots,l$, $j = 1,\dots,M$, and $k = 1,\dots,N$ are of collocation type, we are able to accelerate the evaluation of the deformations and stresses with the introduced AMVM. 
The evaluation of the stresses using the derivatives $\partial_{x_i} u_h$ $i = 1,2,3$, can be done in a similar way using~\eqref{eq:partiu}.

\section{Numerical Results} \label{sec:numerics}
The numerical experiments are divided into two parts. In both cases the numerical solution of the Lam\'e equations
 \begin{equation} \label{eq:Lame}
 -\mu\, \Delta u(x) - (\lambda + \mu)\, \textnormal{grad} \textnormal{ div } u(x) = 0, \quad x \in \Omega ,
 \end{equation}
 with the Lam\'e constants \eqref{eq:lameconst}
 and $E = 1.0$ (N/mm), $\nu = 0.3$ is computed. The first part deals with the quality of the error estimator in AMVM and the numerical performance of AMVM compared to the multiplication by an approximation obtained from ACA. 
Then, the numerical performance of the combination of AMVM and BACA adapted to linear elasticity (see Sect.~\ref{sec:adapt_BACA}) is investigated in comparison with ACA. First calculations of linear elasticity using the ACA were carried out in~\cite{bg06}.

The computations in this article were performed on a computer with an Intel(R) Core(TM) i7-6700HQ CPU at 2.60 GHz.
All approximation steps in the procedures are performed without parallelization.
The look-ahead approximation is two steps of ACA ahead of the current approximation.

\subsection{Quality of AMVM for linear elasticity}\label{sec:six_one}
The qualitative investigations of AMVM are carried out on three different discretizations of the unit cube $\Omega = [-1,1]^3$ consisting of 488, 1946, and 7778 points.
The following boundary conditions are chosen
 \begin{equation*}
 \gamma_0 u(x) = g_D(x) :=  S(x-p) 
 \end{equation*}
$\textnormal{for } x \in \Gamma_D = \{x \in \Omega:x_1 = 1  \text{ or } x_2 = -1 \text{ or } x_3 = 1\}$ and 
 \begin{equation*}
 \gamma_1 u(x) = g_N(x):= \frac{\partial}{\partial n} S(x-p)  
 \end{equation*}
  $\textnormal{for } x \in \Gamma_N = \{x \in \Omega : x_1 = -1\text{ or } x_2 = 1 \text{ or } x_3 = -1\}$ with $p = (5.0, 5.0, 5.0)^T$.
  We compare the computational time and the storage requirements of AMVM and ACA when computing the
  right-hand side of~\eqref{eq:sysline}. The approximation of the latter will be denoted by~$b_\textnormal{AMVM}$ and~$b_\textnormal{ACA}$, respectively.
 The blockwise accuracy of ACA is chosen to be~$\eps_\textnormal{ACA} = 10^{-6}$ and the admissibility parameter is $\beta = 0.8$.
 The results of ACA are presented in Tables~\ref{tab:LameACA} and~\ref{tab:LameACASpeicher}. 
          \begin{table}[htb] \centering
   \scalebox{1.0}{
  \begin{tabular}{ r  c |  c r }
    $N$ & $b_{\min}$  &  $\norm{b-b_\textnormal{ACA}}_2$ &  time approximation \\ \hline 
  488	&15	&3.45e-7				&6.9	s	\\ [2pt]
  1\,946	&20		&4.43e-7			&39.6 s	\\ [2pt]
  7\,778	&30		&2.53e-7 	 	&248.1 s	\\ 
  \end{tabular}}
  \caption{Error and time required to compute right-hand side of \eqref{eq:sysline} via ACA.} \label{tab:LameACA}
  \end{table}
      \begin{table}[htb] \centering
   \scalebox{1.0}{
  \begin{tabular}{ r |r r r r r r r r r r }
     	& \multicolumn{2}{c}{$V_{\Delta,h}$} 	&  \multicolumn{2}{c}{$V_{11}$}	&  \multicolumn{2}{c}{$V_{12}$}	&  \multicolumn{2}{c}{$V_{13}$}	&  \multicolumn{2}{c}{$V_{22}$} 	\\
       $N$		&MB	& \% 	& MB 	& \%	& MB 	& \%& MB 	& \% 	& MB 	& \%  \\ \hline 
   488	&3.0	&83.5			&3.5	&95.6	&3.5	&96.5	&3.5	&96.4	&3.4	&95.5		\\ [2pt]
  1\,946	&19.7	&34.1			&23.9	&41.4	&23.4	&40.5	&23.3	&40.4	&23.8	&41.3		\\ [2pt]
 7\,778	&115.6 	&12.5			&137.6 	&14.9		&132.7	&14.4	&131.3	&14.2	&137.7	&14.9	\\ 
  \end{tabular}} \\
  \vspace*{0.4cm}
    \scalebox{1.0}{
  \begin{tabular}{ r |r r r r r r r r r r}
     	&  \multicolumn{2}{c}{$V_{23}$}	&  \multicolumn{2}{c}{$V_{33}$}	&  \multicolumn{2}{c}{$K_{\Delta,h}$}\\
       $N$ 	& MB 	& \% 	& MB 	& \% 	& MB 	& \% \\ \hline 
   488	&3.5	&96.1			&3.5	&95.6	&3.6	&99.6		\\ [2pt]
  1\,946	&23.3	&40.4			&23.8	&41.3	&31.1	&53.9		\\ [2pt]
 7\,778	&131.4 &14.2			&136.2 &14.8	&201.1	&21.8	\\ 
  \end{tabular}}
  \caption{Storage requirements for the approximations constructed by ACA.} \label{tab:LameACASpeicher}
  \end{table} 
  
Applying the adaptive matrix-vector multiplication (AMVM) to the linear elasticity problem described in the beginning of Sect.~\ref{sec:numerics} provides for $\theta = 0.7$
the results shown in Tables~\ref{tab:LameAMVM} and~\ref{tab:LameAMVMSpeicher}.
The error $\norm{b-b_\textnormal{AMVM}}_2$ was kept at the same order of magnitude as~$\norm{b-b_\textnormal{ACA}}_2$ in the previous tests.
On all three discretizations of the cube~$\Omega$ a reduction of the computational time could be achieved.
The storage requirements of the operators $V_{\Delta,h}$, $V_{11}$, $V_{12}$, $V_{13}$, 
$V_{22}$, $V_{23}$, and $V_{33}$ turn out to be slightly lower than the corresponding approximations obtained via ACA.
The main benefit is obtained for the operator~$K_{\Delta,h}$.
          \begin{table}[htb] \centering
   \scalebox{1.0}{
  \begin{tabular}{ r   c| c  r }
    $N$ & $b_{\min}$  &  $\norm{b-b_\textnormal{AMVM}}_2$ & time approximation \\ \hline 
  488	&15	&6.34e-7			&5.2	s		\\ [2pt]
  1\,946	&20	&6.85e-7		&29.5 s		\\ [2pt]
  7\,778	&30	&4.63e-7	&188.8	s	\\ 
  \end{tabular}}
  \caption{Error and time required to compute right-hand side of \eqref{eq:sysline} via AMVM.} \label{tab:LameAMVM}
  \end{table}
  
      \begin{table}[htb] \centering
   \scalebox{1.0}{
  \begin{tabular}{ r |r r r r r r r r r r }
      	& \multicolumn{2}{c}{$V_{\Delta,h}$} 	&  \multicolumn{2}{c}{$V_{11}$}	&  \multicolumn{2}{c}{$V_{12}$}	&  \multicolumn{2}{c}{$V_{13}$}	&  \multicolumn{2}{c}{$V_{22}$} 	\\
        $N$ 		&MB	& \%			& MB 	& \%	& MB 	& \% 	& MB 	& \% 	& MB 	& \%  \\ \hline 
   488	&2.9	&79.1			&3.0	&84.2	&3.0	&84.2	&3.0	&84.4	&3.0	&84.2		\\ [2pt]
  1\,946	&19.3	&33.5			&21.3	&37.0	&20.9	&36.2	&20.8	&36.1	&21.3	&36.9		\\ [2pt]
 7\,778	&113.7 &12.3			&120.3 &13.0	&116.6	&12.6	&115.5	&12.5	&120.6	&13.1		\\ 
  \end{tabular}} \\
  \vspace*{0.4cm}
    \scalebox{1.0}{
  \begin{tabular}{ r |r r r r r r r r r r }
     	&  \multicolumn{2}{c}{$V_{23}$}	&  \multicolumn{2}{c}{$V_{33}$}	&  \multicolumn{2}{c}{$K_{\Delta,h}$}\\
       $N$ 		& MB 	& \% 	& MB 	& \% 	& MB 	& \% \\ \hline 
   488	&3.0	&84.3	&3.0	&84.2	&2.2	&60.8	\\ [2pt]
  1\,946	&20.8	&36.1	&21.3	&36.9	&17.2	&29.9		\\ [2pt]
 7\,778	&115.7 &12.5	&119.4 &14.8	&145.4	&15.8	\\ 
  \end{tabular}}
  \caption{Storage requirements for the approximations constructed by AMVM.} \label{tab:LameAMVMSpeicher}
  \end{table}
  
 Before moving on to a more realistic problem, we take a closer look at the reliability and efficiency of the
 error estimator. We present the results obtained in the case of the discretization consisting of
 488~points. We employ a rank-$2$ approximation to start the iterative approximation process.
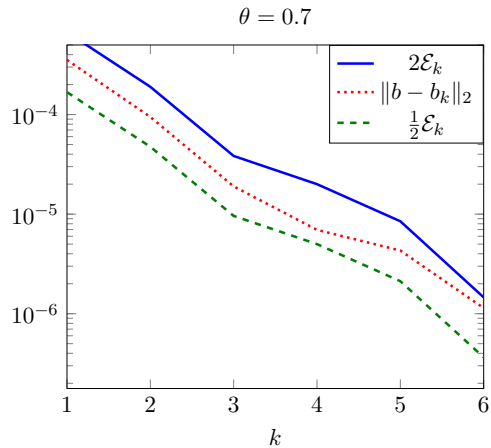
\begin{figure}[H] \centering \scalebox{0.8}{
   \begin{tikzpicture}
   \begin{semilogyaxis}[ 
        xmin = 1, xmax = 6, ymax = 5e-04, xlabel = $k$, 
   legend style = {at={(1,1)}, anchor = north east},
   title={$\theta = 0.7$}]
   \addplot[blue, very thick] coordinates{
   (1,2*3.36e-04)
   (2,2*9.50e-05)
  (3,2*1.92e-05)
   (4,2*1.00e-05)
   (5,2*4.24e-06)
   (6,2*7.30e-07)
   };
   \addlegendentry{$2\mathcal{E}_k$}
   \addplot[red, dotted, very thick] coordinates{
   (1,3.54e-04)
   (2, 9.46e-05)
   (3,1.90e-05)
   (4, 6.95e-06)
   (5, 4.31e-06)
   (6, 1.14e-06)
   };
   \addlegendentry{$\norm{b-b_k}_2$}
     \addplot[green!50!black, dashed, very thick] coordinates{
   (1,0.5*3.36e-04)
   (2,0.5*9.50e-05)
  (3,0.5*1.92e-05)
   (4,0.5*1.00e-05)
   (5,0.5*4.24e-06)
   (6,0.5*7.30e-07)
   };
   \addlegendentry{$\frac{1}{2}\mathcal{E}_k$}
   \end{semilogyaxis}
   \end{tikzpicture}
   }
   \caption{Quality of the error estimator $\mathcal{E}_k$ in the case of AMVM.}
   \label{fig:error_quality_lame}
   \end{figure}
  Figure~\ref{fig:error_quality_lame} shows that the error estimator~$\mathcal E_k$ estimates the error~$\norm{b-b_k}_2$ of the right-hand side reliably and efficiently, which confirms the theoretical results of the adaptive matrix-vector multiplication presented in Sect.~\ref{sec:amvm}.

\subsection{Beam with double-T shape: Load in {$z$}-direction} \label{sec:six_two}
The following experiments focus on the numerical solution of the Lam\'e equations on three discretizations of the geometry shown in Figure~\ref{fig:Mesh1}.
The beam has a length, height and width of~2 with a central part having height and width of~1. 
Figure~\ref{fig:Rand} shows the assignment of the boundary elements to Dirichlet and Neumann part.
On the blue area the beam is loaded with a force of 0.1~N, while the Dirichlet boundary is illustrated by the green area. On the remaining part of the boundary, i.e.\ on the gray area in Figure~\ref{fig:Rand}, 
homogeneous Neumann boundary conditions ($ \gamma_1^\textnormal{int} u(x) = 0$) are prescribed.
The right-hand side of the system of equations which has to be computed is obtained by multiplying the given boundary data by the respective discretized operators $V_h$, $K_h$ and $D_h$; cf.~\eqref{eq:sysline}.

We compare the approximate solution obtained from approximating the coefficient matrix via BACA and ACA, respectively.

\begin{figure}[htb]
\begin{minipage}{0.5\textwidth} \centering
 \scalebox{0.3}{\includegraphics{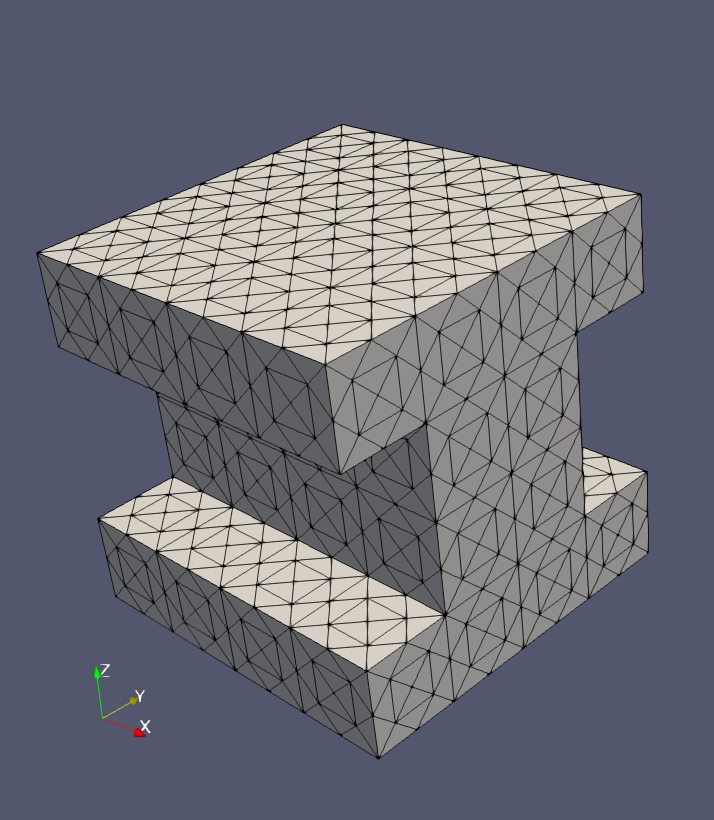}} 
 \caption{Discretization of double T-beam.}\label{fig:Mesh1}
 \end{minipage}
  \hspace*{0.25cm}
\begin{minipage}{0.5\textwidth} \centering \vspace*{0.44cm}
 \scalebox{0.387}{\includegraphics{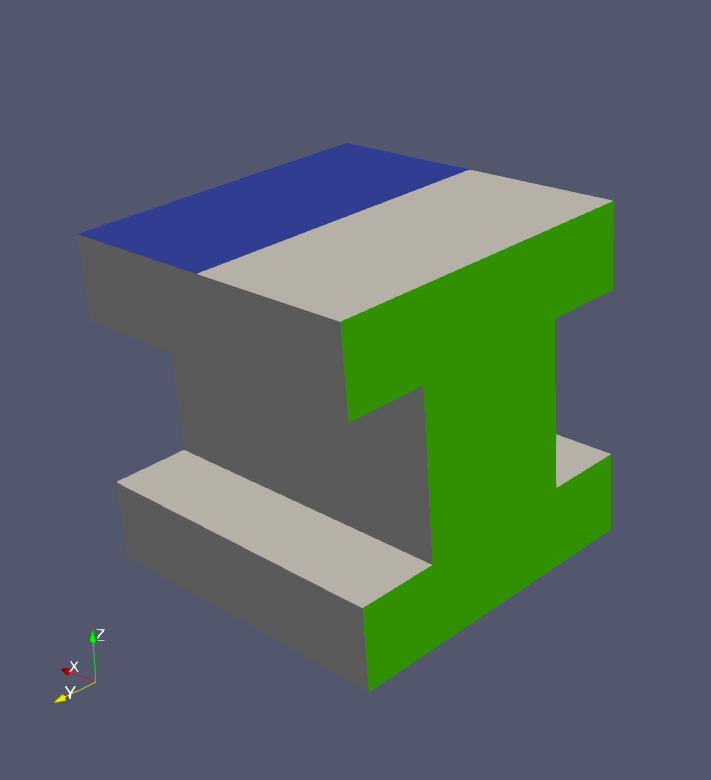}} \vspace*{0.1cm} \\
 \caption{Dirichlet boundary green and loaded Neumann boundary part blue.}\label{fig:Rand}
 \end{minipage}
\end{figure}

The deformations of the beam under load in $z$-direction are shown in Figure~\ref{fig:ACA1}. The maximum absolute differences between the deformations ge-nerated via ACA and BACA in $x$-, $y$- and $z$-direction are $1.2e^{-4}$, $2.4e^{-4}$ and $3.3e^{-4}$. so, both methods ACA and BACA give similar results. 
\begin{figure}[htb]\centering
 \scalebox{0.3}{\includegraphics{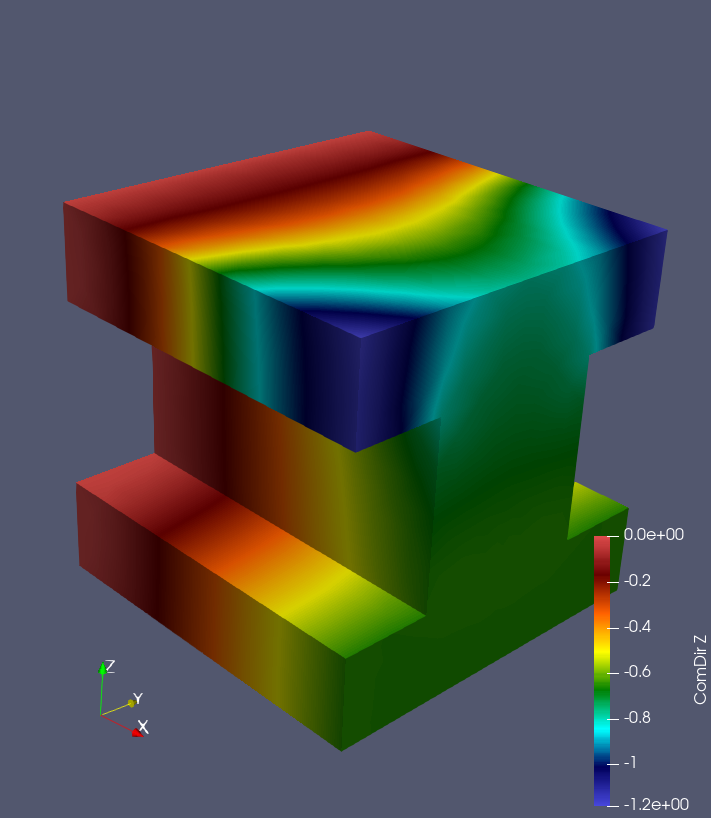}} 
 \caption{Deformation under loading in $z$-direction for ACA and mesh 1.}\label{fig:ACA1}
\end{figure}

The parameters used for ACA in Sect.~\ref{sec:six_one} remain unchanged.
Additionally, we use $\eps_\textnormal{BPCG} = 10^{-5}$ during the iterative
solution via the Bramble-Pasciak conjugate gradients method~\cite{bp88}.
The results for ACA are shown in Table~\ref{tab:Beam_ACA}.

      \begin{table}[H] \centering
   \scalebox{1.0}{
  \begin{tabular}{ r  r |r r r r r r r r r  }
    	&  		& \multicolumn{2}{c}{$V_{\Delta,h}$} 	&  \multicolumn{2}{c}{$V_{11}$}	&  \multicolumn{2}{c}{$V_{12}$}	&  \multicolumn{2}{c}{$V_{13}$}	& $V_{22}$ 	\\
      $N$		&	$M$  		&MB	& \%		 	& MB 	& \%	& MB 	& \%	& MB 	& \% 	& MB 	  \\ \hline 
  1\,664 	&834		&6.9	&65.2			&8.3	&78.3	&8.4	&79.8	&8.3	&78.7	&8.4			\\ [2pt]
  6\,656	&3\,330		&44.4	&26.2			&56.3	&33.3	&56.5	&33.4	&54.6	&32.3	&56.3			\\ [2pt]
 26\,624	&13\,314	&238.7 &8.8			&306.5 	&11.3	&300.8	&11.1	&285.3	&10.6	&305.1			\\ 
  \end{tabular}} \\
  \vspace*{0.4cm}
    \scalebox{1.0}{
  \begin{tabular}{  r  r |r r r r r r r r}
    	&     	& $V_{22}$& \multicolumn{2}{c}{$V_{23}$}	&  \multicolumn{2}{c}{$V_{33}$}	&  \multicolumn{2}{c}{$K_{\Delta,h}$} & time total  \\
      $N$ 		&	$M$ & \%		& MB 	& \% 	& MB 	& \% 	& MB 	& \% \\ \hline 
   1\,664& 834		&79.7&8.4	&79.4			&8.2	&77.2	&9.5		&90.1	&40.3	s	\\ [2pt]
   6\,656	&3\,330		&33.3&54.6	&32.3			&54.0	&31.9	&75.2		&44.9	&426.7	s	\\ [2pt]
  26\,624&13\,314	& 11.3&283.7 &10.5			&287.2 &10.6	&497.3	&18.4	&3\,714.7 s	\\ 
  \end{tabular}}
  \caption{Storage requirements of the approximations constructed via ACA and time consumption of solving the problem.} \label{tab:Beam_ACA}
  \end{table}

For BACA other parameters have to be chosen. The adaptive adjustment of the error tolerance in the Bramble-Pasciak CG is done according to condition~\eqref{eq:solvCond} with $\alpha = 10$.
The initial value of the accuracy in Bramble-Pasciak CG is $10^{-1}$.  The tolerance $\eps_\textnormal{BACA}$ is $10^{-4}$ and $\theta = 0.8$. The starting approximations of the respective $V$ operators are obtained by applying~$8$ (for the two coarsest grids) and $10$ (for the finest grid)
ACA steps. For the operator~$K$ the respective number of steps are~$4$ and~$6$.
Solving the Lam\'e equations via BACA with these parameters leads to the values shown in Table~\ref{tab:Beam_BACA}.

      \begin{table}[H] \centering
   \scalebox{1.0}{
  \begin{tabular}{ r | r |r r r r r r r r r  }
    	&  		& \multicolumn{2}{c}{$V_{\Delta,h}$} 	&  \multicolumn{2}{c}{$V_{11}$}	&  \multicolumn{2}{c}{$V_{12}$}	&  \multicolumn{2}{c}{$V_{13}$}	&  $V_{22}$ 	\\
      $N$  		&	$M$  		&MB	& \%		 	& MB 	& \%	& MB 	& \% 	& MB 	& \% 	& MB  \\ \hline 
  1\,664 	&834		&6.5	&61.0			&6.8	&64.2	&6.8	&64.4	&6.8	&63.9	&6.8			\\ [2pt]
  6\,656	&3\,330		&40.1	&23.7			&41.5	&24.5	&41.2	&24.4	&40.2	&23.8	&41.5			\\ [2pt]
 26\,624	&13\,314	&220.3	&8.1			&228.9 &8.5	&223.8	&8.3	&213.2	&7.9	&228.8			\\ 
  \end{tabular}} \\
  \vspace*{0.4cm}
    \scalebox{1.0}{
  \begin{tabular}{  r | r |r r r r r r r r}
     	&     	& $V_{22}$&  \multicolumn{2}{c}{$V_{23}$}	&  \multicolumn{2}{c}{$V_{33}$}	&  \multicolumn{2}{c}{$K_{\Delta,h}$} & time total  \\
       $N$	&	$M$	& \%	& MB 	& \% 	& MB 	& \% 	& MB 	& \% \\ \hline
   1\,664& 834		& 64.5&6.8	&64.0			&6.8	&63.9	&4.7		&44.4	&30.5 s			\\ [2pt]
   6\,656	&3\,330		& 24.6&40.2	&23.8			&40.5	&24.0	&27.3		&16.1	&215.8 s		\\ [2pt]
  26\,624&13\,314	& 8.5&212.7 &7.9			&218.3 &8.1	&180.4		&6.7	&2\,070.4 s		\\ 
  \end{tabular}}
  \caption{Storage, relative storage for the approximations constructed by BACA and time consumption of solving the problem after applying BACA in the case of Lam\'e equations.} \label{tab:Beam_BACA}
  \end{table}
  Compared to the results obtained from ACA, no significant differences can be observed when applying BACA to the operator~$V_{\Delta,h}$.
  The $V$ operators require only about 70--80\% of the storage needed for
  the approximations ge-nerated via ACA. Stronger benefits can be achieved for
  the~$K_{\Delta}$ operator. Here, the approximation using BACA requires only 50\% (for the coarsest grid) and 36\% (for the two finest grids) of the storage needed in the case of ACA. 
Table~\ref{tab:Beam_BACA} also shows advantages of BACA with respect to computational time.
While on the coarsest grid 75\% of the time need by ACA is consumed,
for the second finest grid the time can be reduced to 51\% and to 56\% for the finest grid considered.
% Finally, computation times can be approximately halved with BACA fitted to the Lame equations on appropriately fine grids.

\section{Conclusion} \label{sec:seven}
In this article, a new method for an adaptive and approximate computation of a matrix-vector multiplication was presented for the case of discretizations of integral operators. The goal was to adapt the approximation to the structure of the vector to be multiplied in order to reduce the storage requirements of the matrix as well as the computational time. Techniques known from adaptive mesh refinement were used in order to identify those blocks which are important for the error of the multiplication.

 After analyzing the convergence of the adaptive method, we focused on the Lamé equations as an application example. Therefore, the adaptation of the new method in the case of linear elasticity was discussed and performed for both
 approximating the system matrix on the left-hand side and approximating the action of operators on the right-hand side. In the numerical examples, the quality of the employed estimator, i.e.\ its reliability and efficiency, could be observed. The application of the
 new methods in case of a loaded beam with double-T shape resulted in less storage requirements and a significant reduction of the computation time compared to solving the considered problem using ACA. 
 
The underlying procedure is also applicable to other problems after minor adaptation as in the case of Lam\'e equations in this article. In future research it will be interesting to see how the algorithms will behave, for example, in the case of Stokes' equations.


\begin{thebibliography}{25}
  
 \bibitem{afl12}
  \small{
  A. Aurada, S. Ferraz-Leite, D. Praetorius,
  Estimator reduction and convergence of adaptive BEM,
  Applied Numerical Mathematics 62 (2012),
  pp. 787--801.
  }
  
      \bibitem{bb21}
      \small{
 M. Bauer, M. Bebendorf,
 Block-Adaptive Cross Approximation of Discrete Integral Operators,
 Computational Methods in Applied Mathematics 21(1),
 2021,
 pp. 13--29.
      }

    \bibitem{bbf21}
    \small{
M. Bauer, M. Bebendorf, B. Feist, 
 Kernel-independent adaptive construction of $\mathcal{H}^2$-matrix approximations,
  Numer. Math. (2021),
   https://doi.org/10.1007/s00211-021-01255-y .
    }

\bibitem{bebendorf00}
  \small{
  M. Bebendorf,
  Approximation of boundary element matrices,
  Numer. Math. 86 (2000),
  pp. 565 -- 589.
  }
  
\bibitem{bebendorf08}
  \small{
  M. Bebendorf,
  Hierarchical Matrices: A Means to Efficiently Solve Elliptic Boundary Value Problems,
  Volume 63 of Lecture Notes in Computational Science and Engineering (LNCSE),
  Springer,
  Berlin(2008).
  } 
  
\bibitem{bg06}
  \small{
  M. Bebendorf, R. Grzhibovskis,
  Accelerating Galerkin BEM for linear elasticity using adaptive cross approximation, 
  Mathematical Methods in the Applied Sciences 29 (2006),
  pp. 1721 -- 1747.
  }
  
\bibitem{br03}
  \small{
  M. Bebendorf, S. Rjasanow,
  Adaptive low-rank approximation of collocation matrices,
  Computing 70 (2003),
  pp. 1 -- 24.
  }
  
    \bibitem{bp88}
      \small{
  J. H. Bramble, J. E. Pasciak,
  A Preconditioning Technique for Indefinite Systems Resulting from Mixed Approximations of Ellipitic Problems,
  Mathematics of Computation 50.181 (1988),
  pp. 1--17.
  }

\bibitem{doerfler96}
  \small{
  W. D\"orfler,
  A convergent adaptive algorithm for Poisson's equation,
  SIAM J. Numer. Anal. 33 (1996),
  pp. 1106--1124.
  }
  
  \bibitem{ey1936}
  \small{
G. Eckart, G. Young,
The approximation of one matrix by another of lower rank,
Psychometrica (1),
pp. 211 -- 218,
1936.
  }
  
\bibitem{fp08}
  \small{
  S. Ferraz-Leite, D. Praetorius,
  Simple a posteriori error estimators for the h-version of the boundary element method,
  Computing 83 (2008),
  pp. 135 -- 162.
  } 
  
    \bibitem{gtz97}
    \small{
  S.A. Goreinov, E.E. Tyrtyshnikov, and N.L. Zamarashkin.
A theory of pseudoskeleton approximations.
 Linear Algebra Appl., 261:1--21, 1997.
    }
  
\bibitem{gh03}
  \small{
  L. Grasedyck, W. Hackbusch,
  Constructions and arithmetics of $\mathcal{H}$-matrices,
  Computing 70 (2003),
  pp. 295 -- 334.
  }
  
\bibitem{gr87}
  \small{
  L.F. Greengard, V. Rokhlin,
  A fast algorithm for particle simulations,
  J. Comput. Phys. 73(2),
  1987,
  pp. 325 -- 348.
  }
  
\bibitem{hackbusch99}
  \small{
  W. Hackbusch,
  A sparse matrix arithmetic based on $\mathcal{H}$-matrices. I. Introduction to $\mathcal{H}$-matrices,
  Computing 62 (1999),
  pp. 89--108.
  }
  
\bibitem{hk00}
  W. Hackbusch, B.N. Khoromskij,
  A sparse $\mathcal{H}$-matrix arithmetic.Part II: Application to multi - dimensional problems,
  Computing 64 (2000),
  pp. 21--47.
  
    \bibitem{han94}
    \small{
  H. Han,
  The boundary integro-differential equations of three dimensional Neumann problem in linear elasticity,
  Numerische Mathematik 68.2,
  1994,
  pp. 269--281.
    }
  
  \bibitem{hi04}
  \small{
  R.B. Hetnarski, J. Ignaczak,
  Mathematical Theory of Elasticity,
  Taylor \& Francis,
  2004.}
  
  \bibitem{rs07}
  \small{
  S. Rjasanow, O. Steinbach,
  The Fast Solution of Boundary Integral Equations,
  Mathematical and Analytical Techniques with Applications to Engineering,
  Springer,
  New York,
  2007.}
  
    \bibitem{r85}
    \small{
  V. Rokhlin,
  Rapid solution of integral equations of classical potential theory,
  J. Comput. Phys., 60(2),
  pp. 187 -- 207,
  1985.}
  
  \bibitem{ss11}
\small {  S. A. Sauter, C. Schwab,
  Boundary Element Methods,
  Springer Series in Computational Mathematics,
  Springer,
  Berlin,
  2011.}
  
\bibitem{os08}
  \small{
  O. Steinbach,
  Numerical Approximation Methods for Elliptic Boundary Value Problems,
  Springer,
  New York (2008).
  }
  
   \bibitem{teodrescu13}
   \small{
  P. P. Teodrescu,
  Treatise on Classical Elasticity. Theory and Related Problems,
  Springer,
  2013.}
 

\end{thebibliography}
\end{document}